\documentclass[12pt]{amsart}
  \usepackage[margin=1in]{geometry}
  \usepackage[english]{babel}
  \usepackage[utf8]{inputenc}
  \usepackage{subcaption}
  \usepackage{amsmath}
  \usepackage{amssymb}
  \usepackage{amsfonts}
  \usepackage{amsthm}
  \usepackage{comment}
  \usepackage{mathrsfs}
  \usepackage[all]{xy}
  \usepackage{capt-of}
  \usepackage[pdftex]{graphicx}
  \usepackage{color}
  \usepackage[noadjust]{cite}
  \usepackage{url}
  \usepackage{indent first}
  \usepackage[labelfont=bf,labelsep=period,justification=raggedright]{caption}
  \usepackage[english]{babel}
  \usepackage[utf8]{inputenc}
  \usepackage{hyperref}
  \usepackage[colorinlistoftodos]{todonotes}
  \usepackage{tkz-fct}
  \usepackage{tikz}
  \usetikzlibrary{calc}
  \usepackage{array}
  \usepackage{booktabs}
  \usepackage{enumerate}
  \renewcommand{\AA}{\mathbb{A}}

  \newcommand{\NN}{\mathbb{N}}

  \newcommand{\QQ}{\mathbb{Q}}
  
  \newcommand{\ZZ}{\mathbb{Z}}

  \newcommand{\mc}[1]{\mathcal{#1}}
  \newcommand{\ol}[1]{\overline{#1}}

  \theoremstyle{plain}
  \newtheorem{thm}{Theorem}
  \newtheorem{lemma}[thm]{Lemma}
  
  \newtheorem{conj}[thm]{Conjecture}
  \newtheorem{prop}[thm]{Proposition}

  \theoremstyle{definition}

  \theoremstyle{remark}
  \newtheorem{rem}[thm]{Remark}

  \numberwithin{equation}{section}
  \numberwithin{thm}{section}

\begin{document}

  \title{The ABC conjecture implies the weak diversity conjecture}
  \author[Hasson]{Hilaf Hasson}
  \address{Hilaf Hasson: University of Maryland, College Park, MD 20742, USA}
  \email{hilaf@math.umd.edu}
  \author[Obus]{Andrew Obus}
  \address{Andrew Obus: University of Virginia, Charlottesville, VA 22904, USA}
  \email{obus@virginia.edu}
  \date{\today}

  \begin{abstract} 
  We show that the $abc$ Conjecture implies the weak diversity conjecture of Bilu and Luca.
  \end{abstract}

  \maketitle

  \section{Introduction}\label{Sintro}
  This note concerns the Weak and Strong Diversity Conjectures. The Strong Diversity conjecture, due to Andrzej Schinzel, first appeared in 
  \cite{DZ}, in the discussion following Theorem 2 of that paper. (The name ``Strong Diversity'' first appeared in 
  \cite{BiluLuca2}, as Conjecture 1.5.)
  \begin{conj}\label{strongdiversity}(``Strong Diversity'')\, 
  Let $X\rightarrow \mathbb{A}^1_{\QQ}$ be a geometrically irreducible branched cover of curves over $\QQ$, such that not all of 
  its branch points are $\QQ$-rational, or such that the cover is not abelian. Let $k(N)$ be the compositum of the fields of rationality of 
  the points in the fibers over $x=1,...,N$. Then there exists a constant $c$, independent of $N$, such that the degree of $k(N)$ over $\QQ$ 
  is at least $e^{cN}$.
  \end{conj}
  We note that the hypotheses in the above conjecture are necessary, and we refer the reader to \cite{DZ} for further discussion. The Strong 
  Diversity Conjecture is closely related to the ``Weak Diversity Conjecture'' (Conjecture 1.4 in \cite{BiluLuca2}).
  \begin{conj}\label{weakdiversity}(``Weak Diversity'')
  Let $K$ be a number field, and let $X\rightarrow \mathbb{A}^1_{K}$ be a non-trivial geometrically irreducible branched cover of curves 
  over $K$. Then there exists a constant $c$ such that the number of different fields appearing as residue fields of the points in the fibers over $x=1,...,N$ is at least $cN$ for all $N$.
  \end{conj}

  We remark that the Weak Diversity Conjecture was only stated in \cite{BiluLuca2} for $K=\mathbb{Q}$, but we, in fact, prove this more 
general form under the assumption of the $abc$ Conjecture. 
Note also that for $K=\QQ$, the consequence of Conjecture \ref{strongdiversity} implies the consequence of Conjecture \ref{weakdiversity}. 
The hypotheses of Conjecture \ref{weakdiversity}, however, are weaker. In \cite{BiluLuca1}, Bilu and Luca prove Weak Diversity 
(for $K = \QQ$) in the case not covered by Strong Diversity, namely for covers where the branch points are $\QQ$-rational, and the 
  cover is abelian. They therefore conclude that Strong Diversity implies Weak Diversity for $K=\QQ$.

\begin{rem}\label{Requivalent}
The Weak and Strong diversity conjectures were stated in \cite{BiluLuca2} in terms of residue fields of a \emph{given} point in each fiber.  
In light of the quantitative version of Hilbert's irreducibility theorem (\cite[Theorem, p.\ 134]{Serre}), all fibers except negligibly many 
have only one point.  So the formulations of \cite{BiluLuca2} are equivalent to our formulations above.  For Weak Diversity, it would also 
be equivalent to look at the \emph{compositum} of the residue fields of all points in each fiber. We use this formulation in Propositions \ref{PGalois} and \ref{Pquotient}.
\end{rem}

  While this was not mentioned in earlier discussions of this conjecture, we remark that the Weak Diversity Conjecture is also closely 
  related to the following conjectural form of a uniform Faltings' Theorem. This form first appeared in \cite{seconduniform}, where Pacelli proves this conjecture under the assumption of Lang's conjecture about rational points on varieties of general type; see also 
  \cite{firstuniform}.

  \begin{conj}\label{uniformfaltings}(``Uniform Faltings' Theorem'')\,
   Let $g\geq 2$ and $d$ be natural numbers. Then there exists a 
  constant $B_{d,g}$ such that for every number field $L$ of degree $d$ over $\QQ$, and for every curve of genus 
  $g$ over $L$,  we have that $X(L)\leq B_{d,g}$.
  \end{conj}

  As we will soon see (Proposition \ref{PGalois}), the Weak Diversity Conjecture can be reduced to case of $G$-Galois covers $f: X \to \AA^1$. In the Galois case, 
  Conjecture \ref{uniformfaltings} implies Weak Diversity for $g(X)\geq 2$. Indeed, for a right $G$-torsor $T$ over $K$, 
  there exists a twist $X^T$ of $X$ such that for $K$-rational points $P$ of $\mathbb{A}^1_K$, the restriction $X \times_{\AA^1} \{P\}$ is isomorphic 
  to $T$ as a right $G$-torsor iff $X^T$ has a $K$-rational point above $P$. See, for example, Lemma 3.3.1 of \cite{hilafmin}, and 
  surrounding discussion. Since all of these twists have the same genus, we are done. In this way, Weak Diversity can be viewed as a 
   weaker form of Conjecture \ref{uniformfaltings} that, unlike Conjecture \ref{uniformfaltings}, also applies to genera $0$ and $1$. 
  Note that Conjecture \ref{uniformfaltings} is not even known for twists of a given curve; see related results in this direction in 
  \cite{jacobiansilver} and \cite{jacobianstoll}.

  Strong Diversity is known in either of two cases: (a) when one of the branch points is of degree either $2$ or $3$ above $\QQ$ 
  (\cite[Theorem 2(b)]{DZ}), or (b) if the branch points are all $\QQ$-rational and the normal closure of $X\rightarrow \mathbb{A}^1_{\QQ}$ 
  satisfies some condition (for example if its Galois group is either alternating, symmetric or non-abelian simple group of non-square 
order; 
  see \cite{DZ2}). Weak Diversity (but not Strong Diversity) was also proven (\cite[Corollary 1]{CZ}) in the case that $X$ has at least $3$ 
  geometric points above $\infty$. See also 
  Proposition \ref{PCorvajaZannier}, and preliminary discussion thereof, in this paper.

  In this paper we prove that the $abc$ Conjecture (for an appropriate number field) implies Weak Diversity (Theorem \ref{Tweakdiversity}), 
and that $abc$ implies Strong Diversity for the case that not all branch points are $\QQ$-rational (Theorem \ref{Plargedegree}). We also, 
  unconditionally, reduce Weak Diversity to the cyclic Galois case.

	We mention that Mochizuki claims to have proven the Vojta conjecture for all curves over number fields (\cite[Discussion after Theorem A]{Mochizuki}), which implies the $abc$ Conjecture over number fields.  If Mochizuki's proof is verified, then Weak Diversity will hold unconditionally.

\section*{Acknowledgements}

   The authors thank Larry Washington for fruitful conversations, and Andrew Granville, Ram Murty and Taylor Dupuy for very 
thorough and helpful answers to their mathematical inquiries.

\section{Proof of the non-rational branch point case of Strong Diversity given $abc$}\label{Sstrong}

As was mentioned above, Dvornicich and Zannier proved Strong Diversity for $f: X \to \AA^1_{\QQ}$ whenever $f$ has a branch point of index $2$ or $3$.  Combining the $abc$ Conjecture with a result of Granville allows us to weaken this assumption to $f$ having a branch point not defined over the base field.

\begin{lemma}\label{Lsquarefree}
  Assume the $abc$ Conjecture.  Then
  $$n=O(\#\{p \geq n \mid v_p(g(m)) = 1 \text{ for some } m \leq n \} )$$ whenever $g \in \ZZ[x]$ is an irreducible polynomial of degree at 
  least $2$.  If $\deg g \in \{2,3\}$, then the $abc$ Conjecture is not required.
  \end{lemma}

  \begin{proof}
  By \cite[Eq.\ (1) on p.\ 427]{DZ}, the lemma is true unconditionally if $v_p(g(m)) = 1$ is replaced by $p \mid g(m)$.  So it suffices to 
  show that 
  $$\#\{p \geq n \mid v_p(g(m)) > 1 \text{ for some } m \leq n \} = o(n).$$
  
  If $\deg g \in \{2, 3\}$, this 
  follows as on \cite[p.\ 427]{DZ}, without the $abc$ Conjecture.  In any case, if $\deg g \geq 3$, this follows from \cite[Theorem 
8]{Granville} applied to the homogenization of $g$, taking $N = 
  n$ and $M = 1$.
  \end{proof}

  \begin{thm}\label{Plargedegree}
  Suppose that the branch locus $\Delta$ of $f: X \to \AA^1_{\QQ}$ contains a point of degree $\geq 2$ over $\QQ$, and that the $abc$ Conjecture is true.  Then Strong Diversity holds for $f$. 
  \end{thm}

  \begin{proof}
  Let $X'$ be a plane curve such that $X \dashrightarrow X' \stackrel{f'}{\to} \AA^1_{\QQ}$ is a factorization of $f$ as a rational map 
with $X \dashrightarrow X'$ birational. To prove Strong Diversity for $f$, it suffices to prove it for $f'$. 
  
  Since $X'$ is a plane curve, we are in the situation of \cite{DZ}.  If $\Delta$ has a point of degree $2$ or $3$ over $\QQ$, then this is \cite[Theorem 2(b)]{DZ}.  The only input to the proof in \cite{DZ} that requires $\Delta$ to have a point of degree $2$ or $3$ is the result of Lemma 
  \ref{Lsquarefree} for some irreducible factor $g$ of a polynomial cutting out $\Delta$ (see \cite[(11), p.\ 437]{DZ}).  By our assumptions 
on 
  $\Delta$, there is such a factor of degree $\geq 2$.  Since we assume the $abc$ Conjecture, the proposition follows from Lemma 
  \ref{Lsquarefree}.
  \end{proof}

  \section{Unconditional reduction of Weak Diversity to cyclic case}\label{Sreduction}
  
  In this section, we reduce Weak Diversity to the case of cyclic covers of prime order.  We do not assume the $abc$ Conjecture.  
   
\begin{lemma}\label{Lbasechange}
If a cover $f: X \to \AA^1_{K_0}$ is defined over $K_0$, then Weak Diversity for $f$ is equivalent to Weak Diversity for any base change $f_K$ over a number field extension $K/K_0$.
\end{lemma}

\begin{proof}
The residue field of a point in $f_K^{-1}(n)$ is the compositum of the residue field of the corresponding point of $f^{-1}(n)$ with $K$. If 
two number fields have distinct composita with $K$, they must be distinct.  On the other hand, there are only finitely many distinct number 
fields whose composita with $K$ are identical.  The lemma follows.
\end{proof}

  \begin{prop}\label{PGalois}
  Suppose that $f: X \to \AA^1_K$ is a cover defined over $K$ and   
  $L/K$ is a finite extension for which the Galois closure $f': X' \to \AA^1_L$ of the base-change $f_L$ of $f$ to $L$ is geometrically irreducible and defined over $L$ as a Galois cover. Then to prove Weak Diversity for $f$, it suffices to prove it for $f'$.
  \end{prop}

  \begin{proof}
By Lemma \ref{Lbasechange}, we may assume that $L = K$ and $f_L = f$.
Let $L_n$ (resp.\ $L'_n$) be the field generated by the residue fields of the points of $f^{-1}(n)$ (resp.\ $(f')^{-1}(n)$).   We note 
that $L'_n$ is Galois over $K$ and is contained in the Galois closure of $L_n$ over $K$.  So $L'_n$ is the Galois closure of $L_n$ over $K$.  So if $L'_i \neq L'_j$, then $L_i \neq L_j$.  Thus Weak Diversity for $f'$ implies Weak Diversity for $f$.
  \end{proof}

  \begin{prop}\label{Pquotient}
  Suppose $f: X \to \AA^1_K$ is a quotient cover of $g: Y \to \AA^1_K$.  Then Weak Diversity is true for $g$ if it is true for $f$.
  \end{prop}

  \begin{proof}
  Let $L_n$ (resp.\ $L'_n$) be the field generated by the residue fields of the points of $f^{-1}(n)$ (resp.\ $g^{-1}(n)$).  
  Then $L_n \subseteq L'_n$ and the degree of $L'_n$ over the base field is bounded in terms of $g$, which means that there exists $d \in \NN$ such that each $L'_i$ can correspond to at most $d$ non-isomorphic 
  $L_j$’s.  So if the number 
  of distinct $L'_n$ for $n \leq N$ is at least $cN$, then the number of distinct $L_n$ for $n \leq N$ is at least $cN/d$. 
  \end{proof}

  Proposition \ref{PCorvajaZannier} below was stated in \cite{BiluLuca2} as a consequence of \cite[Corollary 1]{CZ}, but we supply some more details on the proof here.  Recall that if $L$ is a number field and $S$ is a finite set of places containing the archimedean places, then $\mc{O}_{L,S} \subset L$ is the subring of $L$ consisting of elements whose valuations at all places outside of $S$ are nonnegative.



  \begin{prop}\label{PCorvajaZannier}

Let $f: X \to \AA^1_{K_0}$ be a branched cover defined over a number field $K_0$. If the smooth projective completion of $f$ has at 
  least three $\ol{\QQ}$-points over $\infty$, then Weak Diversity holds for $f$.   
  \end{prop}

  \begin{proof}
 Embed $X \subset \AA^m_{K_0}$ as an affine curve. If $K/K_0$ is a finite extension and $S$ is a finite set of places of $\mc{O}_K$ including the archimedean places, \cite[Corollary 1]{CZ} implies that the number of $\mc{O}_{K,S}$-integral points of $X$ is bounded in terms of the degree of $K$ and the cardinality of $S$.  Now, since the ring extension $K_0[X]/K_0[t]$ corresponding to $f$ is generated by roots of finitely many monic polynomials over $K_0$, there is a finite set of places $S_0$ of $K_0$ such that the same is true for $\mc{O}_{K_0, S_0}[X]/\mc{O}_{K_0, S_0}[t]$.  Taking $S$ to be the set of places of $K$ lying above $S_0$, we see that every $K$-point of $X$ lying above an $\mc{O}_{K_0, S_0}$-point of $\AA^1_{\mc{O}_{K_0, S_0}}$ is in fact an $\mc{O}_{K,S}$-point.  Thus, the number of such points is bounded solely in terms of the degree of $K$. 

  Since any field $L$ arising as the residue field of a point of $f^{-1}(n)$ for $n \in \NN$ has degree at most $\deg(f)$ over $K_0$, there is an absolute bound, depending only on $f$, on the number of such points with residue field $L$.  This 
  immediately implies Weak Diversity for $f$.
  \end{proof}

  \begin{prop}\label{Pcyclic}
  To prove Weak Diversity for a cover defined over a number field with a given branch locus $\Delta$, it suffices to prove it 
  for cyclic covers of prime order with branch locus contained in $\Delta$.
  \end{prop}

  \begin{proof}
  By Lemma \ref{Lbasechange} and Proposition \ref{PGalois}, we may assume the cover is Galois for some group $G$.  If the cover has at 
least three $\ol{\QQ}$-points 
  defined over $\infty$, then the proposition follows from Proposition \ref{PCorvajaZannier}, so assume there are at most two such points.  
  Then the stabilizer of one of these points is a cyclic group of index at most $2$ in $G$.  So either $G$ is cyclic or $G$ has $\ZZ/2$ as a 
  quotient.  In either case, $G$ has a cyclic group of prime order as a quotient, and the quotient cover has branch locus contained in 
  $\Delta$, so we are done by Proposition \ref{Pquotient}.  
  \end{proof}

\begin{rem}
The most difficult case for the Weak Diversity Conjecture seems to be that of a \emph{quadratic} cover.  In this case, it is tantamount to showing that for a separable polynomial $f \in K[x]$, the number of distinct square classes in the set $\{f(1), \ldots, f(N)\}$ is at least $cN$ for some constant $c > 0$ and all $N$.
\end{rem}

\section{Proof of Weak Diversity given $abc$}

\begin{lemma}\label{Lnorm}
Let $K$ be a number field, and let $f(x) \in \mc{O}_K[x]$ be a non-constant polynomial.  Then there is a constant $c$, depending on $f$, such that for any ideal $I \subseteq \mc{O}_K$, the set $\{n \in \NN \mid (f(n)) = I\}$ has cardinality bounded by $c$.
\end{lemma}

\begin{proof}
It suffices to bound the number of $n$ such that $N_{K/\QQ}(f(n))$ equals any particular constant.  But $N_{K/\QQ}(f(n))$ is a polynomial in $n$ over $\QQ$, whose absolute value is easily seen to go to $\infty$ as $n \to \infty$.  Thus it is non-constant, and the lemma follows.
\end{proof}

\begin{thm}\label{Tweakdiversity}
Let $f: X \to \AA^1_K$ be a geometrically 
irreducible branched cover over some number field $K$, and let $L$ be a number field such that each branch point of $f$ is $L$-rational.  Then the $abc$ Conjecture for $L$\footnote{See, e.g., \cite[p.\ 84]{Vojta}} implies Weak Diversity holds for $f$.
\end{thm}

  \begin{proof}
  
By Lemma \ref{Lbasechange} we may, without loss of generality, assume that $L=K$.
By Proposition \ref{Pcyclic}, we may assume that $f$ is a $\ZZ/p$-cover, for some prime $p$.  After a base change, and using Lemma 
\ref{Lbasechange} again, we may assume that $f$ is given by an equation $y^p = g(x)$, where $g(x) \in \mc{O}_K[x]$ is a polynomial with 
roots exactly at the branch points and all roots of $g(x)$ have order at most $p-1$.  

Let $h(x) \in \mc{O}_K[x]$ be a separable polynomial with the same leading coefficient and roots as $g(x)$.  By the number field version\footnote{See the remark on \cite[p.\ 993]{Granville}} of \cite[Theorem 1]{Granville}, there exists a positive constant $c$ and an ideal $I \subseteq \mc{O}_K$ such that for large enough $N$, the ideal $(h(n))I^{-1} \subseteq \mc{O}_K$ is squarefree for at least $cN$ elements $n \in \{1, \ldots, N\}$.  By Lemma \ref{Lnorm}, after replacing $c$ by a smaller positive constant, we can find $cN$ elements $n \in \{1, \ldots, N\}$ such that $(h(n))I^{-1}$ is squarefree and the ideals $(h(n))$ are pairwise distinct.  After replacing $c$ by yet a smaller constant, we may assume that the prime factorizations of the ideals $(h(n))$ are pairwise distinct even when prime factors of $I$ and of $(p)$ are ignored.  

Now, $h(n) \mid g(n) \mid h(n)^{p-1}$, so the primes ramified in $K(g(n)^{1/p})/K$, other than those dividing $I$ or $(p)$, are exactly those primes dividing $(h(n))$.  Thus the fields $K(g(n)^{1/p})$ are pairwise distinct, which proves Weak Diversity for $f$.
\end{proof}

\begin{rem}
Combining Theorem \ref{Tweakdiversity} with Theorem \ref{Plargedegree}, we see that assuming the $abc$ Conjecture over $\QQ$ suffices to prove Weak Diversity for covers defined over $\QQ$, even if the branch locus does not consist of $\QQ$-points. 
\end{rem}

  \bibliographystyle{alpha}
  \bibliography{main.bib}

  \end{document}